\documentclass{article}
\usepackage[english]{babel}
\usepackage{amsrefs}
\usepackage{amssymb,amsmath,amsthm,amsfonts,epsfig,graphicx,psfrag,color}
\usepackage[utf8]{inputenc}
\usepackage{amsrefs}
\usepackage{xcolor}
\usepackage{floatrow}
\usepackage{wrapfig}
\usepackage{subfigure}
\usepackage{caption}
\usepackage{float}
\usepackage{hyperref}
\usepackage{datetime}

\newdateformat{monthyeardate}{\monthname[\THEMONTH], \THEYEAR}
\hypersetup{
    colorlinks,
    citecolor=black,
    filecolor=blue,
    linkcolor=black,
    urlcolor=blue
}

\DeclareMathOperator{\Obs}{Obs}

\DeclareMathOperator{\per}{Per}

\newcommand{\W} {\mathcal W}

\DeclareMathOperator{\card}{card}
\theoremstyle{plain}
\newtheorem{thm}{Theorem}[section]
\newtheorem{cor}[thm]{Corollary}
\newtheorem{lem}[thm]{Lemma}
\newtheorem{prop}[thm]{Proposition}

\theoremstyle{definition}
\newtheorem{rmk}[thm]{Remark}
\newtheorem{ex}{Example}[section]
\newtheorem{df}{Definition}[section]

\theoremstyle{remark}
\newtheorem{que}{Question}[section]

\DeclareMathOperator{\diam}{diam}
\DeclareMathOperator{\len}{len}

\DeclareMathOperator{\dist}{dist}

\DeclareMathOperator{\ord}{ord}

\newcommand{\R}{\mathbb R}

\newcommand{\C}{\mathbb C}
\newcommand{\Z}{\mathbb Z}
\newcommand{\N}{\mathbb N}

\newcommand{\T}{\mathbb T}
\newcommand{\Suno}{\mathbb{S}^1}
\renewcommand{\epsilon}{\varepsilon}

\begin{document}

\author{M. Achigar, A. Artigue\footnote{Email: artigue@unorte.edu.uy. Adress: Departamento de Matemática y Estadística del Litoral, Universidad de la Rep\'ublica, Gral. Rivera 1350, Salto, Uruguay.}\; and I. Monteverde}
\title{Observing Expansive Maps}
\date{\monthyeardate\today}
\maketitle

\begin{abstract}
 We consider the problem of the observability of positively expansive maps by the time series associated to continuous real functions. For this purpose we prove a general result on the generic observability of a locally injective map of a compact metric space of finite topological dimension, extending earlier work by Gutman \cite{Gut16}. We apply this result to partially solve the problem of finding the minimal number of functions needed to  observe a positively expansive map. We prove that two functions are necessary and sufficient for positively expansive maps on tori.
\end{abstract}

\section{Introduction}
 Suppose that the states of a natural system are modeled by a compact metric space $(X,\dist)$. If we consider discrete time then the evolution of the model may be given by a continuous map $T\colon X\to X$,  where if $x\in X$ is an initial state then $Tx\in X$ represents the state of the system after a certain fixed time interval has elapsed. We say that a continuous function  $f\colon X\to\R$ \emph{observes} $T$ if $x\neq y$ implies $f(T^nx)\neq f(T^ny)$ for some $n\geq 0$. The function $f$ represents a measurement of the system and is called an \emph{observable} \cite{Gut16} or \emph{output function} \cite{Nerurkar}. A classical result on this topic, proved by Aeyels \cite{Ae} and Takens \cite{Takens}, essentially says that a generic $C^2$-function $f\colon X\to\R$ observes a $C^2$-diffeomorphism $T\colon X\to X$ of a smooth manifold $X$. 
 
 Several authors considered this problem from a topological viewpoint, for a continuous map $T$ on a finite dimensional compact metric space. 
 Jaworski \cite{Coorn}*{Theorem 8.3.1} proved that every homeomorphism $T$ without periodic points is observed by a generic $f$. In \cite{Nerurkar}*{Theorem 4.1} Nerurkar showed an analogous result but allowing finitely many points of each period. Recently, Gutman \cite{Gut16} proved that an injective map $T$ satisfying fewer hypothesis on periodic points (explained in Theorem \ref{teoGutmanHomeo}) is generically observed.

 In this paper we consider this problem for non-injective maps. In our main result, Theorem \ref{thmGutmanEndos}, we extend Gutman's Theorem for locally injective maps. The main difficulty for the observability of a non-injective map is simple to state: there may not be enough time to observe a pair of \emph{collapsing} points, i.e., two points $x\neq y$ such that $T^nx=T^ny$ for some $n\geq 1$. In Theorem \ref{thmGutmanEndos} we prove a generic observability criterion, but some pairs of collapsing points must be excluded. It is easy to see that for $Tz=z^2$ on the unit complex circle it is impossible to observe every pair of collapsing points. See Lemma \ref{lemaGrTop} for a generalization. In Proposition \ref{propObsNoGen} we show that for a non-injective map $T$ of a compact manifold the set of functions that observes $T$ is not dense, in particular, for these maps observability is not generic. The problem presented by collapsing points motivates to consider more observations, i.e., a function $f\colon X\to\R^m$. In Corollary \ref{corObsMuerto1} we show that if a function $f\colon X\to \R^m$ observes every pair of collapsing points then a small perturbation of $f$ observes every pair of distinct points.

 In \S \ref{secAppExpSys} we show that there are strong connections between the topics of observability and expansivity. We say that $T$ is \emph{positively expansive} if there is $\delta>0$ such that if $x\neq y$ then $\dist(T^nx,T^ny)>\delta$ for some $n\geq 0$. These maps have several important properties that will be exploited in this paper. For example, they cannot be injective \cite{CK} (see also \cite{AAM}) and if a compact metric space admits a positively expansive map then the space has finite topological dimension \cites{Ma,AH,Kato93}. Also, if the space is a compact manifold then every positively expansive map is conjugate to an infra-nilmanifold expanding endomorphism, see \cites{Hi88,Shub69}. In particular, positively expansive maps on tori are conjugate to linear expanding endomorphisms. For our purposes, it is remarkable that positively expansive maps satisfy the hypothesis of Theorem \ref{thmGutmanEndos}, see Proposition \ref{propObsPosExp}.
 
 Given a map $T\colon X\to X$ we introduce the \emph{observability number} of $T$, defined as the minimal natural number $m$ for which there is a function $f\colon X\to \R^m$ that observes $T$. The observability number will be denoted as $\Obs(T)$. In Theorem \ref{teoObsLinTn} we show that if $T\colon \T^d\to \T^d$ is a hyperbolic toral endomorphism then: $\Obs(T)=1$ if $T$ is invertible and $\Obs(T)=2$ if $T$ is non-invertible. As a consequence, in Corollary \ref{con2alcanza} we deduce that for every positively expansive map $T$ on a torus it holds that $\Obs(T)=2$.
 
 We also introduce the concept of \emph{strict observability}, motivated as follows. 
 Given that measurements in reality always 
 have an error, say $\epsilon>0$, it is natural to consider observations as different when the evaluations of the function $f\colon X\to\R^m$ are away from the fixed precision $\epsilon$. That is, we will require that $\|f(T^nx)-f(T^ny)\|>\epsilon$ for some $n\geq 0$, whenever $x\neq y$, and in this case we say that \emph{$f$ strictly observes $T$}. In Proposition \ref{propObsPosExp} we show that for positively expansive maps observability is equivalent to strict observability. 

 The authors thank Damián Ferraro for useful conversations in the initial stage of the present research and acknowledge his idea of linking expansivity with a strong form of observability, 
 the one we call strict observability. In Proposition \ref{propStObsPosExp} we show that positive expansivity is equivalent to strict observability. This supports the statement: \emph{we can only observe expansive systems}.
  
\section{Observability}\label{secObs}
 In this section we state Theorem \ref{thmGutmanEndos} and derive some consequences. Let $(X,\dist)$ be a compact metric space. Denote by $C(X,\R^m)$ the set of continuous functions $f\colon X\to\R^m$ endowed with the usual compact-open topology. For $m=1$ set $C(X)=C(X,\R)$. 

 In this paper we will consider the \emph{topological dimension} as defined in \cite{HW}. It will be denoted as $\dim X$. We will recall a special result for our purposes, which also characterizes the topological dimension. Suppose $f\in C(X,\R^m)$. A point $y\in f(X)$ is called an \emph{unstable value} of $f$ if for all $\delta>0$ there is $g\in C(X,\R^m)$ such that $\|f(x)- g(x)\| < \delta$ for every $x\in X$ and $y\notin g(X)$. Other points of $f(X)$ are called \emph{stable values}.
 
\begin{thm}[\cite{HW}*{Theorems VI.1 and VI.2 on pp.\;75 -- 77}]\label{teoCharTopDim}
 For $X$ a compact metric space the following statements are equivalent: 
 \begin{enumerate}
  \item $\dim X\leq d$,
  \item all the values of $f$ are unstable, for all $f\in C(X,\R^{d+1})$. 
 \end{enumerate}
\end{thm}

\begin{rmk}\label{obsDefTopDim}
 The concept of topological dimension has several 
 different equivalent definitions. In light of Theorem \ref{teoCharTopDim}, $\dim X=d$ if and 
 only if $d$ is the minimal number for which all the values of $f$ are unstable, 
 for all $f\in C(X,\R^{d+1})$. 
 We note that the (equivalent) definition used by Mañé 
 and Kato \cites{Ma,Kato93} for the study of expansive dynamics and by Gutman \cite{Gut15}
 is the one called \emph{covering dimension} \cite{HW}.
\end{rmk}

 Let $T\colon X\to X$ be a continuous map. We say that $T$ is \emph{locally injective} if for all $x\in X$ there is $\epsilon>0$ such that the restriction of $T$ to the ball $B_\epsilon(x)$ is injective. For a function $f\in C(X,\R^m)$ and $n\in\N$ we define $f_0^n\colon X\to\R^{m(n+1)}$ as 
  $$f_0^n(x)=\bigl(f(x),f(Tx),\ldots,f(T^nx)\bigr).$$
 The function $f_0^n$ represents a sequence of $n+1$ observations, usually called \emph{time series}. 
 
\begin{df}
 Given a set $W\subseteq X\times X$ and $n\in\N$, we say that $f\in C(X,\R^m)$ \emph{observes $T$ on $W$ in $n$ steps} if $f_0^n(x)\neq f_0^n(y)$ for all $(x,y)\in W$. For the case $W=X\times X\setminus\Delta$, where $\Delta=\{(x,x):x\in X\}$, we simply say that \emph{$f$ observes $T$ in $n$ steps.} 
\end{df}

 This definition represents the idea that the time series distinguishes pairs of points of $W$ in a bounded time interval. Given $n\in\N$ define the following sets 
  $$\Delta_n(T)= \bigl\{(x,y)\in X\times X:T^nx=T^ny\bigr\},$$
 and for $n\geq1$
  $$\per_n(T)=\{x\in X:T^jx=x\text{ for some }1\leq j\leq n\}.$$

 Now we can state our main result.		
 
\begin{thm}\label{thmGutmanEndos}
 Let $T\colon X\to X$ be a locally injective continuous map on a compact metric space $X$ of $\dim X\leq d$ such that $\dim \per_n(T)<n/2$ for $n=1,\ldots,2d$. Then the set 
  $$\Omega=\bigl\{f\in C(X):f\text{ observes }T\text{ on }X\times X\setminus \Delta_{2d}(T)\text{ in }2d\text{ steps}\bigr\}$$
 is residual.
\end{thm}

 When a property holds in a residual set we say that the property is \emph{generic}. Therefore, Theorem \ref{thmGutmanEndos}  means that a generic function of $C(X)$ observes a map $T$ on $X\times X\setminus \Delta_{2d}(T)$ in $2d$ steps. The proof of Theorem \ref{thmGutmanEndos} is given in \S \ref{secGutEnd}. The next result says that Theorem \ref{thmGutmanEndos} is an extension of \cite{Gut16}*{Theorem 1.1}.

\begin{thm}[Gutman's Theorem]\label{teoGutmanHomeo}
 Let $T\colon X\to X$ be an injective continuous map on a compact metric space $X$ of $\dim X\leq d$ such that $\dim \per_n(T)<n/2$ for $n=1,\ldots,2d$. Then the set 
  $$\Omega=\{f\in C(X):f\text{ observes }T\text{ in }2d\text{ steps}\}$$
 is residual.
\end{thm}

\begin{proof}
 For an injective map $T$ it holds that $\Delta_n(T)=\Delta=\{(x,x):x\in X\}$ for all $n\in\N$. Then, the result follows by Theorem \ref{thmGutmanEndos}.
\end{proof}

 We proceed to explain what can happen if $T$ is not injective and to derive some consequences of Theorem \ref{thmGutmanEndos}.

\begin{df}
 We say that \emph{$f$ observes $T$} if for all $x\neq y$ there is $n\geq 0$ such that $f(T^nx)\neq f(T^ny)$.
\end{df}

 The next result shows that for a non-injective map in the hypothesis of Theorem \ref{thmGutmanEndos} it could happen that the set of functions $f\in C(X)$ that observes $T$ is not residual. Moreover, we show that the set of functions $f\in C(X,\R^d)$ that observes $T$ is not even dense, when $X$ is a compact manifold of $\dim X=d\geq1$.

\begin{prop}\label{propObsNoGen} 
 Let $X$ be a compact $d$-dimensional manifold, $d\geq1$. If $T\colon X\to X$ is a non-injective, locally injective continuous map then there exists an open set $U\subseteq C(X,\R^d)$ such that no $f\in U$ observes $T$.
\end{prop}

\begin{proof}
 Take $x_1,x_2\in X$ such that $x_1\neq x_2$ and $Tx_1=Tx_2=y$. Consider two disjoint compact neighborhoods $A_i$ of $x_i$, $i=1,2$, and $B$ of $y$ such that each $T|_{A_i}\colon A_i\to B$ is a homeomorphism. Let $h\colon A_1\to A_2$ be the homeomorphism $h=T|_{A_2}^{-1}\circ T$. Since $X$ is a $d$-dimensional manifold we have $\dim A_1=d$, and then, by Theorem \ref{teoCharTopDim}, there exists a continuous function $g\colon A_1\to\R^d$ such that $0$ is a stable value of $g$. Take $f\in C(X,\R^d)$ such that $g(x)=f(x)-f\bigl(h(x)\bigr)$ for $x\in A_1$ (for example choose $f=g$ on $A_1$, $f=0$ on $A_2$ and extend using Tietze's extension theorem). Then any sufficiently small perturbation of $f$ gives rise to a small perturbation of $g$ which takes $0$ as a value. Thus the perturbations $\bar f$ of $f$ will take the same value at some $a\in A_1$ and $b=h(a)\in A_2$, and we have $Ta=Tb$ by the definition of $h$. This proves that no $\bar f$ in a neighborhood of $f$ observes $T$.
\end{proof}

 There are functions that observe but do not observe in any bounded number of steps. Let us give an example. 

\begin{ex}
 On the circle $\Suno=\{z\in\C:|z|=1\}$ consider the map $T\colon \Suno\to \Suno$ defined as $Tz=z^2$. Let $h\colon [0,2\pi]\to [0,2\pi]$ be such that $h(0)=h(\epsilon)=0$, $h(2\pi)=2\pi$ and extended linearly, and define $f\colon \Suno\to \C$ by $f(e^{\theta i})=e^{h(\theta)i}$. It is easy to see that if $\epsilon$ is small then $f$ observes $T$. Since $f$ is constant in an arc containing the fixed point $z=1$, it does not observe in any bounded number of steps.
\end{ex}

 The next corollary is important since it reduces the problem of finding an observing function $f$ for a given system to the problem of finding a function $f$ distinguishing collapsing points. For this purpose, given $n\geq 1$, we define
  $$\Delta^{*}_n(T)=\bigl\{(x,y)\in X\times X:x\neq y,\ldots, T^{n-1}x\neq T^{n-1}y\text{ and } T^nx=T^ny\bigr\}.$$
 Note that for $f\in C(X,\R^m)$ the condition that $f$ observes $T$ on $\Delta_1^*(T)$ means that $f(x)\neq f(y)$ if $x\neq y$ and $Tx=Ty$. 

\begin{cor}\label{corObsMuerto1}
 Let $T\colon X\to X$ be a locally injective continuous map on a compact metric space $X$ of $\dim X\leq d$ such that $\dim \per_n(T)<n/2$ for $n=1,\ldots,2d$. Then, if there exists $f\in C(X,\R^m)$ observing $T$ on $\Delta^*_1(T)$ then there exists a perturbation $\bar f\in C(X,\R^m)$ of $f$ observing $T$ in $2d$ steps.
\end{cor}

\begin{proof}
 Suppose that $f\in (X,\R^m)$ observes $T$ on $\Delta_1^*(T)$. Then, it is clear that $f$ observes $T$ on $\Delta_1^*(T)$ in $0$ steps and on $\Delta^{*}_n(T)$ in $n-1$ steps for all $n\geq1$. Therefore $f$ observes $T$ on $\Delta_{2d}(T)\setminus\Delta=\bigcup_{n=1}^{2d}\Delta^{*}_n(T)$ in $2d-1$ steps. Since $T$ is locally injective, we have that $\Delta^*_1(T)$ is compact and therefore any sufficiently small perturbation $\bar f\in C(X,\R^m)$ of $f$ will observe $T$ on $\Delta^*_1(T)$ and on $\Delta_{2d}(T)\setminus\Delta$ in $2d-1$ steps. Now applying Theorem \ref{thmGutmanEndos} we see that we can choose a perturbation $\bar f\in C(X,\R^m)$ of $f$ which simultaneously observes $T$ on $\Delta_{2d}(T)\setminus\Delta$ in $2d-1$ steps and on $X\times X\setminus \Delta_{2d}(T)$ in $2d$ steps. Then we conclude that $\bar f$ observes $T$ in $2d$ steps.
\end{proof}

 In the next section we give more applications of these results in the study of expansive maps.

\section{Strict observability and expansive maps}\label{secAppExpSys}

 In this section we introduce the concept of strict observability for the study of positively expansive maps. It is remarkable that positive expansivity is equivalent to strict observability. We also prove that the observability number of a positively expansive map on a torus equals 2. 

 Let $T\colon X\to X$ be a continuous map of a compact metric space $(X,\dist)$.

\begin{df}\label{dfPosExp}
 We say that $T$ is \emph{positively expansive} if there is $\delta>0$ such that if $x\neq y$ then $\dist(T^nx,T^ny)>\delta$ for some $n\geq 0$. Such $\delta$ is called an \emph{expansivity constant}.
\end{df}

 It is known that if a compact metric space admits a positively expansive map then $\dim X<\infty$. The proof is analogous to the case of expansive homeomorphisms, see \cite{Ma} (also \cites{AH,Kato93}).

\begin{df}
 We say that $f\in C(X,\R^m)$ \emph{strictly observes} the map $T$ if there exists $\epsilon>0$ such that for all $x\neq y$ there exists $n\geq 0$ such that $\|f(T^nx)-f(T^ny)\|>\epsilon$. 
 In this case we say that $\epsilon$ is a \emph{precision} of the observation.
\end{df}

\begin{prop}\label{propObsPosExp}
 Let $T\colon X\to X$ be a positively expansive map on a compact metric space $X$ of $\dim X\leq d$ and $m\geq 1$. 
 Then:
 \begin{enumerate}
  \item a generic function in $C(X,\R^m)$ observes $T$ on $X\times X\setminus \Delta_{2d}(T)$,
  \item if $f\in C(X,\R^m)$ observes $T$ then $f$ strictly observes $T$.
 \end{enumerate}
\end{prop}

\begin{proof}
 The first part follows by Theorem \ref{thmGutmanEndos} since for every positively expansive map it is easy to see that $T$ is locally injective and $\per_n(T)$ is a finite set for all $n\geq 1$.

 To prove the second part, we argue by contradiction. Suppose that $f\in C(X,\R^m)$ observes $T$ but  for each $n\geq 1$ there are $x_n,y_n\in X$ such that $x_n\neq y_n$ and $\|f(T^kx_n)-f(T^ky_n)\|<1/n$ for all $k\geq 0$. Let $\delta>0$ be an expansivity constant of $T$ and take $m_n\geq 0$ such that $\dist(T^{m_n}x_n,T^{m_n}y_n)>\delta$ for all $n\geq 1$. Since $X$ is compact (taking a subsequence) we can assume that $T^{m_n}x_n\to x$ and $T^{m_n}y_n\to y$. Then $\dist(x,y)\geq\delta$ and $x\neq y$. But $f(T^kx)=f(T^ky)$ for all $k\geq 0$. Since $f$ observes $T$ we have a contradiction and the proof ends.
\end{proof}

 For future reference we state the following classical result.

\begin{thm}[\cite{HW}*{Theorem V 2, p. 56}]\label{teoHWEncaje}
 If $X$ is a compact metric space with $\dim X\leq d$ then a generic function $f\in C(X,\R^{2d+1})$ is an embedding.
\end{thm}

\begin{prop}\label{propStObsPosExp}
 For a continuous map $T\colon X\to X$ of a compact metric space of $\dim X\leq d$, the following statements are equivalent:
\begin{enumerate}
 \item $T$ is positively expansive,
 \item a generic function $f\in C(X,\R^{2d+1})$ strictly observes $T$,
 \item there is $f\in C(X,\R^m)$ that strictly observes $T$, for some $m\geq 1$.
\end{enumerate}
\end{prop}

\begin{proof}
 \mbox{($1\to2$)} From Theorem \ref{teoHWEncaje} we know that a generic map $f\in C(X,\R^{2d+1})$ is injective. In particular, such $f$ observes $T$ and by Proposition \ref{propObsPosExp}, $f$ strictly observes $T$.
 \mbox{($2\to3$)} It is obvious.
 \mbox{($3\to1$)} Suppose that $f\in C(X,\R^m)$ strictly observes $T$ with precision $\epsilon$. Since $X$ is compact and $f$ is continuous there is $\delta>0$ such that if $\dist(x,y)<\delta$ then $\|f(x)-f(y)\|<\epsilon$. Then, $\delta$ is an expansivity constant.
\end{proof}

\begin{df}
 The \emph{observability number} of $T$ is the minimal $m\geq 1$ for which there is $f\in C(X,\R^m)$ that observes $T$. This number will be denoted as $\Obs(T)$. 
\end{df}

\begin{rmk}
 A general bound of the observability number is 
  $$\Obs(T)\leq2\dim X+1.$$
 This bound follows by Theorem \ref{teoHWEncaje}. In particular, if $\dim X=0$ then $\Obs(T)=1$ for every map $T\colon X\to X$. 
\end{rmk}

 We will calculate the observability number of some classes of dynamical systems. We start with the case of homeomorphisms. Recall that a \emph{continuum} is a compact connected set.

\begin{df}
 A homeomorphism $T\colon X\to X$ is \emph{expansive} if there is $\delta>0$ such that if $x\neq y$ then $\dist(T^nx,T^ny)>\delta$ for some $n\in\Z$. We say that the homeomorphism $T$ is \emph{cw-expansive} if there is $\delta>0$ such that if $C\subset X$ is a non trivial continuum then $\diam(T^nC)>\delta$ for some $n\in\Z$.
\end{df}

 It is clear that every cw-expansive homeomorphism is an expansive homeomorphism. The concept of expansivity is well known in the context of hyperbolic diffeomorphisms, as for example basic sets of Smale's Axiom A diffeomorphisms and Anosov diffeomorphisms. The idea of cw-expansivity was introduced by Kato \cite{Kato93}. His deep understanding of Mañé's arguments in \cite{Ma} allowed him to extend, with simpler proofs, some results for expansive homeomorphisms to cw-expansivity. For our purposes, it is remarkable that if a compact metric space admits a cw-expansive homeomorphism then the space has finite topological dimension \cite{Kato93}*{Theorem 5.2}.

\begin{prop}\label{propCwObs1}
 If $T$ is a cw-expansive homeomorphism of a compact metric space $X$ then $\Obs(T)=1$. Moreover, a generic function $f\in C(X)$ observes $T$.
\end{prop}

\begin{proof}
 It is known \cite{Kato93} that if a compact metric space admits a 
 cw-expansive homeomorphism then $\dim X<\infty$. 
 It is easy to see that for a cw-expansive homeomorphism it holds that 
 $\dim \per_k(T)=0$ for all $k\geq 1$.
 Then, the result follows by Gutman's Theorem.
\end{proof}

Now we give a general estimate that 
generalizes the example of $z\mapsto z^2$ in the unit complex circle given in the introduction.

\begin{lem}\label{lemaGrTop}
 Let $G$ be a compact connected topological group and consider a continuous endomorphism $T\colon G\to G$. 
 If $T$ is non-injective then $\Obs(T)\geq 2$.  
\end{lem}

\begin{proof}
 Let $f\in C(G)$. Since $T$ is non-injective there is $x\neq e$, where $e$ denotes the identity of $G$, such that $Tx=e$. Since $\int_Gf(y)-f(xy)\,d\mu(y)=0$, where $d\mu(y)$ denotes the right invariant Haar measure of $G$, we have $f(y)-f(xy)=0$ for some $y\in G$ by the continuity and connectedness assumptions. Then $y\neq xy$, $f(y)=f(xy)$ and $Tx=Txy$. 
\end{proof}

Let $\T^d=\R^d/\Z^d$ be the $d$-dimensional torus. 
Given an invertible matrix $A\in M_d(\Z)$ we say that the endomorphism $T\colon \T^d\to \T^d$ given by $Tx=Ax$ is \emph{hyperbolic} if $A$ has no eigenvalues of modulus one. The class of hyperbolic toral endomorphisms includes:
 \begin{enumerate}
  \item \emph{Anosov diffeomorphisms}, when $\left|\det(A)\right|=1$,
  \item \emph{Expanding endomorphisms}, if all the eigenvalues of $A$ have modulus greater than one,
  \item \emph{Anosov endomorphisms}, if it is not invertible and $A$ presents expanding and contracting eigenvalues.
 \end{enumerate}
The following matrices represents each of these classes in the two-dimensional torus:
 \[
  \left(
   \begin{array}{ll}
    2 & 1 \\ 1 & 1
   \end{array}
   \right),\,\,
     \left(
   \begin{array}{ll}
    2 & 0 \\ 0 & 2
   \end{array}
   \right),\,\,
  \left(
   \begin{array}{ll}
    3 & 1 \\ 1 & 1
   \end{array}
   \right)
  \]
  respectively.
 See \cite{AH} for more on this topic.

\begin{thm}\label{teoObsLinTn}
If $T$ is a hyperbolic toral endomorphism then: 
 \begin{enumerate}
  \item $\Obs(T)=1$ if $T$ is invertible (Anosov diffeomorphism) and 
  \item $\Obs(T)=2$ if $T$ is non-invertible (expanding or Anosov endomorphism).
 \end{enumerate}
\end{thm}

\begin{proof} 
 As we said, if $T$ is invertible then it is an Anosov diffeomorphism. In particular, it is expansive and Proposition \ref{propCwObs1} implies that $\Obs(T)=1$.

 Assume that $T$ is non-invertible. Then by Lemma \ref{lemaGrTop} $\Obs(T)\geq2$. To obtain the reverse inequality we will show that $T$ satisfies the hypothesis of Corollary \ref{corObsMuerto1} for a certain function $f\colon\T^d\to\R^2$. Firstly, note that as $A$ is invertible, $T$ is locally injective. Given that $T$ is hyperbolic, 1 is not an eigenvalue of $A^k$ and therefore $\per_k(T)$ is finite (and zero-dimensional) for all $k\geq 1$. 
 
 Finally, we now define a function $f\colon\T^d\to\C$ observing $T$ on $\Delta_1^*(T)$. To this end view $\T^d=\Suno\times\cdots\times\Suno$ as a multiplicative group, where we consider $\Suno\subset\C$. 
 Note that as $T$ is locally injective, the set $D=\ker T$ is finite. Let $m_1,\ldots,m_d>0$ be such that if $a=(a_1,\ldots,a_d)\in D\setminus\{(1,\ldots,1)\}$ and $a_i$ is the last coordinate not equal to $1$ then $|1-a_i|>m_i$. Let $\rho_1,\ldots,\rho_d>0$ defined inductively as
 \begin{equation*}
  \left\{\begin{array}{l}
         \rho_1=1,\\
         \rho_im_i=1+2\sum_{k=1}^{i-1}\rho_k,\quad i=2,\ldots,d.
         \end{array}\right.
 \end{equation*}
 Note that by definition we have:\quad$2\sum_{k=1}^{i-1}\rho_k<\rho_im_i$\quad for $i=1,\ldots,d$.
 
 Define $f\colon\T^d\to\C$ as 
  $$f(x_1,\ldots,x_d)=\sum_{k=1}^n\rho_kx_k\qquad\text{for}\quad(x_1,\ldots,x_d)\in\T^d.$$
 To see that $f$ observes $T$ on $\Delta_1^*(T)$, consider $(x,y)\in\Delta_1^*(T)$, that is $x=(x_1,\ldots,x_d)$ and $y=(y_1,\ldots,y_d)$ such that $x\neq y$ and $Tx=Ty$. Note that the element $a=xy^{-1}$ satisfies $a\in D$ and $a\neq(1,\ldots,1)$. Let $a_i$ be the last coordinate of $a$ not equal to $1$. By definition of $m_1,\ldots,m_d$ we have $|1-a_i|>m_i$. Now we estimate
 \begin{equation*}
  \begin{split}
   |f(y)-f(x)|&=\Bigl|\sum_{k=1}^n\rho_ky_k-\sum_{k=1}^n\rho_kx_k\Bigr|
     =\Bigl|\sum_{k=1}^n\rho_k(y_k-x_k)\Bigr|
     =\Bigl|\sum_{k=1}^i\rho_k(y_k-x_k)\Bigr|\\
     &\geq\rho_i|y_i-x_i|-\sum_{k=1}^{i-1}\rho_k|y_k-x_k|
     \geq\rho_i|1-a_i||y_i|-2\sum_{k=1}^{i-1}\rho_k\\
     &>\rho_im_i-2\sum_{k=1}^{i-1}\rho_k>0.
  \end{split}
 \end{equation*}
 Therefore $f(x)\neq f(y)$, and then $f$ observes $T$ on $\Delta_1^*(T)$. As we had verified all the hypothesis of Corollary \ref{corObsMuerto1} we conclude that a perturbation of the given $f$ observes $T$, thus $Obs(T)\leq2$.\end{proof}
 
 \begin{rmk}
 In the proof of the second assertion of the previous lemma, we only use that $1$ is not an eigenvalue of $A^k$ for $k=1,\ldots,2d$. Then, it can be generalized for such invertible matrices. For example 
 \[A=\left(
  \begin{array}{rr}
    0 & 1\\ -1 & 1
  \end{array}
  \right).\]
\end{rmk}

\begin{cor}\label{con2alcanza}
If $T\colon \T^d\to\T^d$ is a positively expansive map then $\Obs(T)=2$.
\end{cor}

\begin{proof}
From \cites{Hi88,Shub69} we know that $T$ is conjugate to an expanding endomorphism. 
Then, the result follows by Theorem \ref{teoObsLinTn}.
\end{proof}

\begin{que}
 Is it true that $\Obs(T)=2$ for every positively expansive map of a compact metric space $X$ with $\dim X\geq 1$?
\end{que}

 To end this section we make some remarks about the observability of expansive homeomorphisms and positively cw-expansive maps. The proofs are analogous to our previous arguments and are therefore omitted.

\subsubsection*{Observability of expansive homeomorphisms}
 Let $T\colon X\to X$ be an expansive homeomorphism of the compact metric space $(X,\dist)$. If $f\in C(X,\R^m)$ strictly observes $T$ then $T$ is positively expansive, recall Proposition \ref{propStObsPosExp}. Applying \cite{CK} we conclude that $X$ is a finite set. Then, it is natural to introduce the following definition. We say that $f\in C(X,\R^m)$ \emph{strictly bilaterally observes $T$} if there is $\epsilon>0$ such that if $x\neq y$ then $\|f(T^nx)-f(T^ny)\|>\epsilon$ for some $n\in\Z$. 
 Analogous to Propositions \ref{propObsPosExp}, \ref{propStObsPosExp} and \ref{propCwObs1} we have: 

\begin{prop}
 Let $T\colon X\to X$ be an expansive homeomorphism of a compact metric space $X$. Then, if $f\in C(X,\R^m)$ observes $T$ then $f$ strictly bilaterally observes $T$.
\end{prop}

\begin{prop}
 For a homeomorphism $T\colon X\to X$ of a compact metric space and $m\geq1$ the following statements are equivalent:
\begin{enumerate}
 \item $T$ is expansive,
 \item a generic function $f\in C(X,\R^m)$ strictly bilaterally observes $T$.
 \item there exists $f\in C(X,\R^m)$ that strictly bilaterally observes $T$.
\end{enumerate}
\end{prop}

\subsubsection*{Observability of positively cw-expansive maps}
 A map $T\colon X\to X$ is \emph{positively cw-expansive} \cite{Kato93} if there is $\delta>0$ such that if $C\subset X$ is a non-trivial continuum then $\diam(T^nC)>\delta$ for some $n\geq 0$. If $T\colon X\to X$ is a positively cw-expansive map then: $\dim\per_k(T)=0$ for all $k\geq 1$ and $\dim X<\infty$. We deduce the following result.
 \begin{prop} 
  Let $X$ be a compact metric space with $\dim X\leq d$ and $T\colon X\to X$ a positively cw-expansive map. If $T$ is locally injective then a generic function in $C(X)$ observes $T$ on $X\times X\setminus \Delta_{2d}(T)$.
\end{prop}
 
 We give an example showing that a positively cw-expansive map may not be locally injective. 

\begin{ex}\label{exDosCirculos}
 Consider the circles $X_1=\{z\in\C:|z|=1\}$ and $X_2=\{z\in\C:|z-2|=1\}$. Define $X=X_1\cup X_2$ and $T\colon X\to X$ by 
 $$Tz=\left\{
  \begin{array}{ll}
   z^2&\text{ if }z\in X_1,\\
   2-z&\text{ if }z\in X_2.\\
  \end{array}
 \right.$$
 Note that $TX_2=X_1$ and that $T$ is positively cw-expansive on $X_1$. Also, we have that $T$ is not locally injective at $z=1$.
\end{ex}

\section{Proof of Theorem \ref{thmGutmanEndos}}\label{secGutEnd}
 This section is devoted to the proof of Theorem \ref{thmGutmanEndos}. As we said, it extends Gutman's Theorem weakening the injectivity assumption on the map $T$ to local injectivity. 
  
 For the reader familiar with the techniques of \cites{Gut15, Gut16} let us remark some differences with   our approach that gives a shorter proof. On one hand, we do not partitioned $X\times X\setminus\Delta$ into product subspaces. Instead we consider other subspaces, not necessarily disjoint, and glue them  together using Lemma \ref{quelindelof}. For example  we consider the subspaces $G_{k,l}$, roughly speaking the graph of $T^k$ (see Definition \ref{dfGkl}). On the other hand, to simplify some arguments, we consider the definition of topological dimension based on unstable values (recall Remark \ref{obsDefTopDim}). For example, we use this definition in the proof of Lemma \ref{ObsUV_grafica}. We apply this lemma to manage the cases of pairs of points that after some iterates are in a common positive orbit. In order to distinguish a periodic (or preperiodic) and a non-periodic point we introduce Lemma \ref{ObsUV_disjuntosNPvsP}, whose proof is based on Proposition \ref{haylugar} which is again a result inspired by the ideas around the definition of topological dimension based on unstable values. Finally, Lemma \ref{ObsUV_disjuntosNP} takes care of the simplest case: two non-periodic points not in the same orbit. With these lemmas we solve all the cases arising in the proof of Theorem \ref{thmGutmanEndos}. 
      
 We start with a generalization of the direct part of Theorem \ref{teoCharTopDim}.
  
\begin{lem}\label{esquivaAff}
 Let $U$ be a compact metric space of $\dim U\leq r$ and $H_1,\ldots ,H_n\subseteq\R^{r+s+1}$ affine subspaces such that $\dim H_i\leq s$, $i=1,\ldots,n$. Then the set
  $$\Omega=\{F\in C(U,\R^{r+s+1}):F(U)\cap H_i=\varnothing\text{ for all }i=1,\ldots,n\}$$
 is open and dense.
\end{lem}
\begin{proof}
 Clearly, it suffices to consider the case $n=1$. Let $H=H_1$. As $H$ is closed and $U$ is compact we have that $\Omega$ is open. Performing a translation of coordinates if necessary we may suppose that $H$ is a linear subspace. Let $H^{\perp}$ be the orthogonal complement of $H$. Given a function $F\in C(U,\R^{r+s+1})$ decompose $F=F_H+F_{H^{\perp}}$, where $F_H\colon U\to H$ and $F_{H^{\perp}}\colon U\to H^{\perp}$ are the compositions of $F$ with the orthogonal projections on $H$ and $H^{\perp}$ respectively. Since $\dim H\leq s$ we have $\dim H^{\perp}\geq r+1$.  Then, given that $\dim U\leq r$, by 
 Theorem \ref{teoCharTopDim} we can perturb $F_{H^{\perp}}$ to get $\widetilde F_{H^{\perp}}\in C(U,H^{\perp})$ so that $0\notin\widetilde F_{H^{\perp}}(U)$. Then taking $\widetilde F=F_H+\widetilde F_{H^{\perp}}$ we have that $\widetilde F$ is a perturbation of $F$ and $\widetilde F(U)\cap H=\varnothing$. This proves that $\Omega$ is dense.
\end{proof}

 The next result implies that two finite dimensional metric spaces immersed in $\R^n$ are generically disjoint if $n$ is big enough.

\begin{prop}\label{haylugar}
 Let $U,V$ be compact metric spaces with $\dim U +\dim V\leq n$ and $S\subset\R^{n+1}$ a convex set. Then the set
  $$\Omega=\{(F,G)\in C(U,\R^{n+1})\times C(V,S):F(U)\cap G(V)=\varnothing\}$$
 is open and dense in $C(U,\R^{n+1})\times C(V,S)$.
\end{prop}

\begin{proof}
 As $U$ and $V$ are compact, we have that $\Omega$ is open. To prove that it is dense, let $r=\dim U$ and $s=\dim V$. Take $(F,G)\in C(U,\R^{n+1})\times C(V,S)$. As $\dim V=s$ we can perturb $G$ to get $\widetilde G$ as in the proof of \cite{HW}*{(C) on pp. 57-59} taking a finite cover $\beta$ of $V$ of $\ord(\beta)\leq s$ consisting of small open balls, for each $W\in\beta$ a point $p_W\in G(W)\subseteq S$ and letting $\widetilde G(y)=\sum_{W\in\beta}\omega_W(y)p_W$ for $y\in V$, where $\{\omega_W:W\in\beta\}$ is a partition of the unity subordinate to $\beta$. Note that as $p_W\in S$ for $W\in\beta$ and $S$ is convex we have $\widetilde G(V)\subseteq S$, because $G(y)$ is a convex linear combination of the points $p_W$ for each $y\in V$. Moreover, as $\ord(\beta)\leq s$, each $\widetilde G(y)$ is the convex linear combination of at most $r+1$ points $p_W$. Then each $\widetilde G(y)$ lies in one of the affine subspaces of dimension less o equal to $s$ generated by $s+1$ points $p_W$. That is, $\widetilde G(V)\subseteq\bigcup_iH_i$, for a finite family $\{H_i\}$ of affine subspaces with $\dim H_i\leq s$. Then, as $\dim U=r$ and $r+s\leq n$, by Lemma \ref{esquivaAff} we can perturb $F$ to get $\widetilde F$ so that $\widetilde F(U)$ does not meet $\bigcup_iH_i$, and consequently $\widetilde F(U)\cap\widetilde G(V)=\varnothing$. This proves that $\Omega$ is dense. 
\end{proof}

\begin{rmk}
 For the case $\dim U=n$, $V=\{p\}$ and $S=\{0\}$, Proposition \ref{haylugar} becomes equivalent to \cite{HW}*{Theorem VI.1 on p.\;75}.
\end{rmk}

 A proof of the following lemma can be found in \cite{Gut15}*{Lemma A.5}.

\begin{lem}\label{tietze}
 Let $X$ be a normal topological space, $f\colon X\to\R$ a bounded continuous function, $A\subseteq X$ a closed subset and $g_0\colon A\to\R$ a bounded continuous function such that $d(f|_A,g_0)<\epsilon$, where $\epsilon>0$. Then there exists a continuous and bounded extension $g\colon X\to\R$ of $g_0$ such that $d(f,g)<\epsilon$.
\end{lem}

The following result is generalization of \cite{Coorn}*{Lemma 8.3.4}.

\begin{lem}
\label{ObsUV_disjuntosNP}
 Let $X$ be a compact metric space and $T\colon X\to X$ a continuous map. Let $U,V\subseteq X$ be compact subsets with $\dim U\leq n/2$, $\dim V\leq n/2$, $U,\ldots,T^nU,V,\ldots,T^nV$ pairwise disjoint and such that $T$ is injective on the sets $U,\ldots,T^{n-1}U,V,\ldots,T^{n-1}V$. Then the set
  $$\Omega=\{f\in C(X):f\text{ observes }T\text{ on }U\times V\text{ in }n\text{ steps}\}$$
 is open and dense.
\end{lem}

\begin{proof}
 Firstly note that, as $U\times V$ is compact, $\Omega$ is open. To prove that $\Omega$ is dense, consider a function $f\in C(X)$. Let $W=U\cup V$ and $F\colon W\to\R^{n+1}$, $F=f_0^n|_W$. As $\dim W\leq n/2$, by Theorem \ref{teoHWEncaje}, we can perturb $F$ to get an embedding $\widetilde F\in C(W,\R^{n+1})$. Let $\tilde f\colon W_0^n\to\R$ such that $\tilde f(T^ix)={\widetilde F}_i(x)$ if $x\in W$ and $i=0,\ldots,n$, where we write $\widetilde F=(\widetilde F_0,\ldots,\widetilde F_n)$. Then $\tilde f$ is a perturbation of $f|_{W_0^n}$ and by Lemma \ref{tietze} we can extend $\tilde f$ to a perturbation $\bar f\in C(X)$ of $f$. We have that $\bar f_0^n|_W=\widetilde F$, therefore, as $\widetilde F$ is injective, $\bar f$ observes $T$ on $W\times W\setminus\Delta$ in $n$ steps, and in particular on $U\times V$, that is $\bar f\in\Omega$. This proves that $\Omega$ is dense.
\end{proof}

\begin{df}
\label{dfGkl}
 Given a map $T\colon X\to X$ and $k,l\geq0$ we define 
  \begin{eqnarray*}
   G_k=\{(x,T^kx)\in X\times X:x,\ldots,T^kx\text{ are distinct}\}\qquad\text{and}\\
   G_{k,l}=\{p\in X\times X:T^lp\in G_k\text{ and }T^{l-1}p\notin G_k\},
  \end{eqnarray*}
 where we denote $Tp=(Tx,Ty)$ if $p=(x,y)$.
\end{df}

 For a map $T\colon X\to X$, a subset $U\subseteq X$ and $n\geq0$ we define
  $$U_0^n=U\cup TU\cup\cdots\cup T^nU.$$

\begin{lem}\label{ObsUV_grafica}
 Let $X$ be a compact metric space and $T\colon X\to X$ a continuous map. Let $U,V\subseteq X$ be compact subsets with $\dim U\leq n$ such that $T^{l+k}U=T^lV$ and $U,\ldots,T^{l+k-1}U,V,\ldots,T^nV$ are pairwise disjoint, where $0\leq l<l+k\leq n$. Suppose also that $T$ is injective on the sets $U,\ldots,T^{l+k-1}U,V,\ldots,T^{n-1}V$. Then the set
  $$\Omega=\{f\in C(X):f\text{ observes }T\text{ on } U\times V\cap G_{k,l} \text{ in }n\text{ steps}\}$$
 is open and dense.
\end{lem}
\begin{proof}
 Note that by the injectivity assumptions $T$ restricts to homeomorphisms $T^iW\to T^{i+1}W$ for $i=0,\ldots,n-1$ and $W=U,V$. Thus, as $U$ and $V$ are compact we have that $U\times V\cap G_{k,l}$ is compact, and therefore $\Omega$ is open. 
 
 Consider the homeomorphisms $t_i\colon T^iU\to T^iV$, $i=0,\ldots,n$, given by 
  $$t_i=\left\{\begin{array}{ll}
         T^iU\stackrel{T^{l+k-i}}{\longrightarrow}T^{l+k}U=T^lV\stackrel{T^{i-l}}{\longrightarrow}T^iV &\text{if } i< l,\\
         T^iU\stackrel{T^k}{\longrightarrow}T^iV &\text{if } i\geq l,
        \end{array}\right.$$
 and let $W=U_0^n\cup V_0^n$. 
 In Figure \ref{fig:lemay} we illustrate this definition.
 Given a function $f\in C(X)$ define $\delta=(\delta_0,\ldots,\delta_n)\in C(U,\R^{n+1})$ as $\delta_i(x)=f(t_iT^ix)-f(T^ix)$ for $x\in U$ and $i=0,\ldots,n$. We call $\delta$ the \emph{$\delta$-function associated to $f|_W$} (only the
 \thisfloatsetup{capposition=beside, capbesideposition={right,center},floatwidth=.75\textwidth,capbesidewidth=.15\textwidth}
 \begin{figure}[h]
   \centering
   \includegraphics[width=0.75\textwidth]{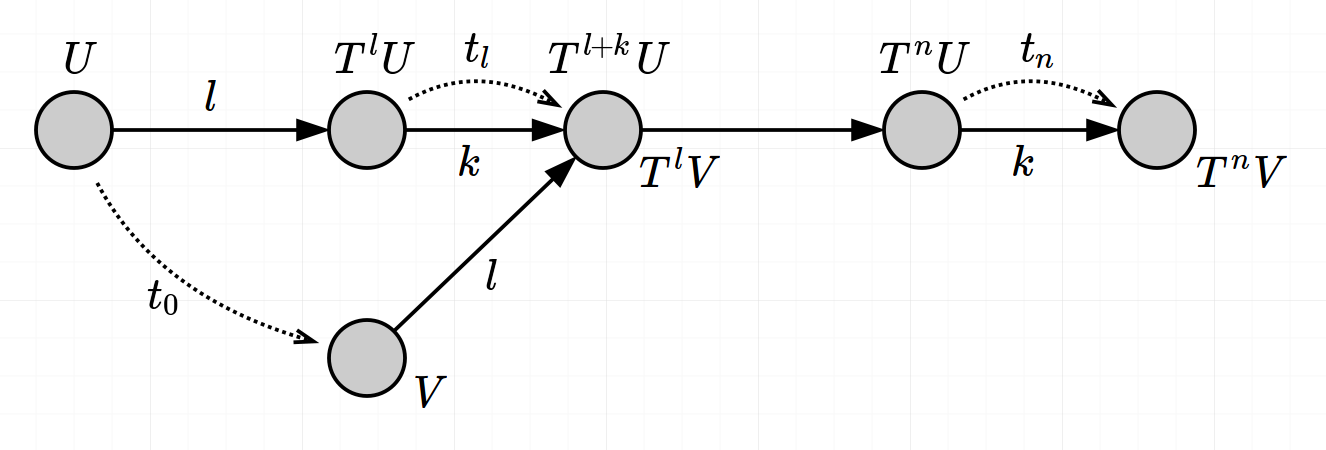}
   \caption{}
   \label{fig:lemay}
 \end{figure}
 values at $W$ are involved). It is not difficult to see that the requirement that $f$ observes $T$ on $U\times V\cap G_{k,l}$ in $n$ steps is equivalent to the condition $0\notin\delta(U)$. If the later is not the case, as $\dim U\leq n$, we can apply Theorem \ref{teoCharTopDim} to get a perturbation $\tilde\delta\in C(U,\R^{n+1})$ of $\delta$ such that $0\notin\tilde\delta(U)$. Define $\tilde f\in C(W)$  as
  $$\tilde f|_{T^iU}=f|_{T^iU},\text{ for }i=0,\ldots,l+k-1,$$
  $$\tilde f|_{T^iV}(t_iT^ix)=f|_{T^iU}(T^ix)+\tilde\delta_i(x),\text{ for }x\in U\text{ and }i=0,\ldots,l,$$
 and on $V_{l+1}^n=T^{l+1}V\cup\cdots\cup T^nV=T^{k+l+1}U\cup\cdots\cup T^{k+n}U$ inductively as
  $$\tilde f|_{T^iV}(T^{k+i}x)=\tilde f|_{T^iU}(T^ix)+\tilde\delta_i(x),\text{ for }x\in U\text{ and }i=l+1,\ldots,n.$$
 One can check that $\tilde f$ is a perturbation of $f|_W$ and that the $\delta$-function associated to $\tilde f$ is precisely $\tilde\delta$. 
 Then, by Lemma \ref{tietze}, we can extend $\tilde f$ to a perturbation $\bar f\in C(X)$ of $f$, which will observe $T$ on $U\times V\cap G_{k,l}$ in $n$ steps because $0\notin\tilde\delta(U)$. This proves that $\Omega$ is dense.
\end{proof}

\begin{df}
 Given a map $T\colon X\to X$ and $n,k\in\N$ we define
  $$H_n=\{x\in X:x=T^nx\text{ and } x,\ldots,T^{n-1}x\text{ are distinct}\},$$
  $$H_{n,k}=\{x\in X:T^kx\in H_n\text{ and }T^{k-1}x\notin H_n\}.$$
 The points $x\in H_{n,k}$ will be called \emph{preperiodic points of period $n$.} 
\end{df}

\begin{lem}\label{ObsUV_disjuntosNPvsP}
 Let $X$ be a compact metric space and $T\colon X\to X$ a continuous map. Let $U\subseteq X$ and $V\subseteq H_{m,k}$ be compact subsets with $\dim U+\dim V\leq n$ such that $U,\ldots,T^nU,V,\ldots,T^{m+k-1}V$ are pairwise disjoint, where $0\leq k\leq m+k\leq n$. Suppose also that $T$ is injective on the sets $U,\ldots,T^{n-1}U,V,\ldots,T^{k-1}V$. Then the set
  $$\Omega=\{f\in C(X):f\text{ observes }T\text{ on }U\times V\text{ in }n\text{ steps}\}$$
 is open and dense.
\end{lem}
\begin{proof}
 Note that as $V\subseteq H_{m,k}$ then $T^kV\subseteq H_m$. 
 Therefore $T$ is injective on the sets $T^kV,\ldots,T^{m+k-1}V$. This together with the injectivity assumptions on $T$ imply that $T$ restricts to homeomorphisms $T^iU\to T^{i+1}U$ for $i=0,\ldots,n-1$ and $T^iV\to T^{i+1}V$ for $i=0,\ldots,m+k-1$.

 Given $f\in C(X)$ let $F\in C(U,\R^{n+1})$ and $G\in C(V,\R^{n+1})$ be given by $F=f_0^n|_U$ and $G=f_0^n|_V$. As $V\subseteq H_{m,k}$ we have 
  $$G(V)\subseteq S=\{(x_0,\ldots,x_n)\in\R^{n+1}:x_i=x_j\text{ if } i=j\;(\textrm{mod }m)\text{ and } i,j\geq k\}$$
 and then $G\in C(V,S)$. As $\dim U + \dim V\leq n$, by Proposition \ref{haylugar} we can perturb $F$ and $G$ to get $\widetilde F\in C(U,\R^{n+1})$ and $\widetilde G\in C(V,S)$ so that $\widetilde F(U)\cap\widetilde G(V)=\varnothing$. Define $\tilde f\in C(U_0^n\cup V_0^{m+k-1})$ by $\tilde f(T^ix)={\widetilde F}_i(x)$ if $x\in U$ and $i=0,\ldots,n$, and $\tilde f(T^iy)={\widetilde G}_i(y)$ if $y\in V$ and $i=0,\ldots,m+k-1$, where we write $\widetilde F=(\widetilde F_0,\ldots,\widetilde F_n)$ and $\widetilde G=(\widetilde G_0,\ldots,\widetilde G_n)$. We have that $\tilde f$ is a perturbation of $f|_W$, where $W=U_0^n\cup V_0^{m+k-1}$. Extend $\tilde f$, applying Lemma \ref{tietze}, to a perturbation $\bar f\in C(X)$ of $f$. Then $\bar f_0^n|_U=\widetilde F$ and, as $\widetilde G(V)\subseteq S$, $\bar f_0^n|_V=\widetilde G$. Thus $\bar f$ observes $T$ on $U\times V$ in $n$ steps, because $\widetilde F(U)\cap\widetilde G(V)=\varnothing$. We conclude that $\Omega$ is dense. Finally, as $U\times V$ is compact $\Omega$ is open.
\end{proof}

\begin{lem}\label{quelindelof}
 Let $M$ be a second-countable space and $\{M_i\}_{i\in I}$ a countable collection of subspaces of $M$. Suppose that for every $p\in M$ there exists $i_p\in I$ and a relative neighborhood $W_p\subseteq M_{i_p}$ of $p$. Then there exists a countable subcover of $\W=\{W_p:p\in M\}$.
\end{lem}
\begin{proof}
 For $i\in I$ let $M^*_i=\{p\in M:i_p=i\}\subseteq M_i$ and note that $\{M^*_i\}_{i\in I}$ is a countable cover of $M$. Let $\W_i=\{W_p:p\in M^*_i\}$. Then $\W_i$ is a cover of $M^*_i$ containing a neighborhood (in $M_i$) of every $p\in M^*_i$. As $M^*_i$ is Lindelöf, there exists a countable subcover $\W^{\circ}_i\subseteq\W_i$ of $M^*_i$. Then $\W^{\circ}=\bigcup_{i\in I}\W^{\circ}_i$ is a countable subcover of $\W$.
\end{proof}

\begin{proof}[Proof of Theorem \ref{thmGutmanEndos}]
 The general strategy of the proof is as follows. Let $M=X\times X\setminus \Delta_{2d}(T)$. For each $p\in M$ we will find a set $W\subseteq M$, relative neighborhood of $p$ in a subspace $M_i\subseteq M$ chosen from a fixed finite collection $\{M_i\}$ of subspaces of $M$, satisfying that
  $$\Omega(W)=\{f\in C(X):f\text{ observes }T\text{ on }W\text{ in }2d\text{ steps}\}$$
 is open an dense. Then, by Lemma \ref{quelindelof}, we can cover $M$ with countable many sets $W$, say with $\{W_k\}_{k\in N}$. Therefore, as $\Omega=\bigcap_{k\in \N}\Omega(W_k)$, we conclude that $\Omega$ is a residual set.
 
 Note that if for $p=(x,y)\in M$ there exists a relative neighborhood $W\subseteq M_i$ as explained before then $W^t=\{(b,a):(a,b)\in W\}$ is a relative neighborhood in $M_i^t$ of $p^t=(y,x)$ satisfying the desired properties and \emph{vice versa}. Then for each $p\in M$ we need to take care only of one of the two points $p$ or $p^t$, by adjoining subspaces $M_i^t$ to $\{M_i\}$  if necessary. That is, we can interchange $x$ and $y$ freely along the proof if needed.
 
 For $z\in X$, we denote with 
 $$z_0^k=\{z,\ldots,T^kz\}$$ 
 the $k$-\emph{step orbit} of $z$ and call \emph{length} of $z_0^k$ the number 
 $$\len(z_0^k)=\card(z_0^k)-1\leq k.$$  
 
 Given $(x,y)\in M$ we distinguish six cases according to whether $x_0^{2d}$ and $y_0^{2d}$ meet or not, 
 and to whether their lengths are equal or less than $2d$. 
 
 We enumerate the cases as in the following table. The remaining cases, corresponding to $\len(y_0^{2d})=2d>\len(x_0^{2d})$ and $2d>\len(y_0^{2d})\geq \len(x_0^{2d})$, are reduced to the cases in the last two rows interchanging $x$ and $y$.
 \newcommand{\XX}{{\LARGE\phantom {X$_{_X}$}}}
 \begin{center}
  \begin{tabular}{c|c|c|}   
   \XX                              &$x_0^{2d}\cap y_0^{2d}=\varnothing$  & $x_0^{2d}\cap y_0^{2d}\neq\varnothing$ \\ \hline
   $\len(x_0^{2d})=\len(y_0^{2d})=2d$     &\XX Case 1 \XX                       &\XX Case 2\XX \\ \hline
   $\len(x_0^{2d})=2d>\len(y_0^{2d})$     &\XX Case 3 \XX                       &\XX Case 4\XX \\ \hline 
   $2d>\len(x_0^{2d})\geq \len(y_0^{2d})$ &\XX Case 5 \XX                       &\XX Case 6\XX \\ \hline
  \end{tabular}
 \end{center}

 {\sc Case 1:} As in this case the $2d$-step orbits are disjoint, both of length $2d$, and $T$ is locally injective we can find compact neighborhoods $U$ and $V$ of $x$ and $y$ respectively, such that $U,\ldots,T^{2d}U,$ $V,\ldots,T^{2d}V$ are pairwise disjoint and $T$ is injective when restricted to $U,\ldots,T^{2d-1}U,$ $V,\ldots,T^{2d-1}V$. Then, as $\dim U\leq d$ and $\dim V\leq d$, we can apply Lemma \ref{ObsUV_disjuntosNP} (with $n=2d$) to conclude that $W=U\times V$ is a neighborhood of $p$ (in the subspace $M_1=M$) for which $\Omega(W)$ is open an dense.
 
 {\sc  Case 2:} In this case both $2d$-step orbits has length $2d$ and meet. As $p\in M$ we have $p\notin \Delta_{2d}(T)$, then $x$ and $y$ do not collapse in the same iterate, so we may suppose that $T^{l+k}x=T^ly$ are the first iterates of $x$ and $y$ that match, with $0\leq l<l+k\leq2d$. Then we are in the situation of Figure \ref{fig:caso2} with $n=2d$.
 \thisfloatsetup{capposition=beside, capbesideposition={right,center},floatwidth=.75\textwidth,capbesidewidth=.15\textwidth}
 \begin{figure}[H]
   \includegraphics[width=.4\textwidth]{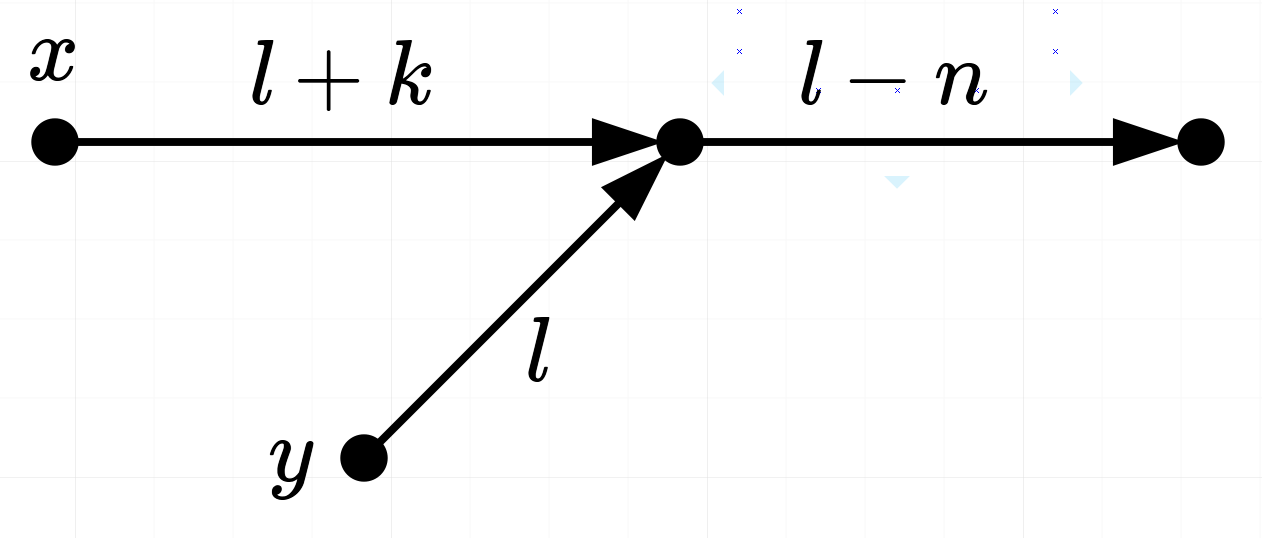}
   \caption{}
   \label{fig:caso2}
  \end{figure}
 As $T$ is a locally injective map there exist compact neighborhoods $U_0$ and $V_0$ of $x$ and $y$ respectively, such that $U_0,\ldots,T^{l+k-1}U_0,$ $V_0,\ldots,T^{2d}V_0$ are pairwise disjoint with $T$ injective on these sets. Let $O=T^{l+k}U_0\cap T^lV_0$, $U=U_0\cap T^{-l-k}O$ and $V=V_0\cap T^{-l}O$. We have that $U$ and $V$ are compact, $T^{l+k}U=T^lV(=O)$, $U,\ldots,T^{l+k-1}U,$ $V,\ldots,T^{2d}V$ are pairwise disjoint and $T$ is injective on these sets. Then, as $\dim U\leq\dim X=d\leq 2d$, we can apply Lemma \ref{ObsUV_grafica} (with $n=2d$) to conclude that $W=U\times V\cap G_{k.l}=U_0\times V_0\cap G_{k.l}$ is a neighborhood of $p$ (in the subspace $M_2=G_{k,l}$) for which $\Omega(W)$ is open and dense. Finally, note that only finitely many subspaces $M_2=G_{k,l}$ are involved here because of the restrictions $0\leq l<l+k\leq2d$.

 {\sc Case 3:} In this case the $2d$-step orbits are disjoint, $\len(x_0^{2d})=2d$ and $\len(y_0^{2d})<2d$. The last condition implies that $y\in H_{m,k}$ with $0\leq k< m+k<2d$. As $T$ is a locally injective map we can take compact neighborhoods $U$ and $V_0$ of $x$ and $y$ respectively, such that $U,\ldots,T^{2d}U,$ $V_0,\ldots,T^{m+k-1}V_0$ are pairwise disjoint and $T$ is injective on these sets. Let $V=V_0\cap H_{m,k}$ and note that $V$ is compact. It is easily checked that the sets $U$ and $V$ are in the hypothesis of Lemma \ref{ObsUV_disjuntosNPvsP} (with $n=2d$), being $\dim U+\dim V\leq 2d$ because $U,V\subset X$ and $\dim X\leq d$. Thus, $W=U\times V$ is a neighborhood of $p$ (in the subspace $M_3=X\times H_{m,k}\setminus \Delta_{2d}(T)$) for which $\Omega(W)$ is open and dense. Finally, observe that only finitely many subspaces $M_3=X\times H_{m,k}\setminus \Delta_{2d}(T)$ are involved in this case because of the restrictions $0\leq k<m+k<2d$.

 {\sc Cases 4 and 6:} These cases are similar to Case 2. In these cases the $2d$-step orbits meet, $\len(y_0^{2d})\leq \len(x_0^{2d})\leq 2d$ and $\len(y_0^{2d})<2d$. As in the previous case, $\len(y_0^{2d})<2d$ implies that $y$ is a preperiodic point, $y\in H_{m,s}$ with $0\leq s< m+s<2d$. In addition, as the $2d$-step orbits meet, $x$ is also a preperiodic point of the same period, $x\in H_{m,r}$ with $0\leq r< m+r\leq2d$. Thus both points $x$ and $y$ reach a common periodic orbit of period $m$ in $r$ and $s$ steps, respectively. 

 We discuss the subcases: (A) $T^rx=T^sy$; (B) $T^rx\neq T^sy$. 

 {\it Subcase A:} In this subcase both orbits meet before they take a step in the common periodic orbit they reach. As in this case $T^rx=T^sy$ we have that $r\neq s$ because $x$ and $y$ do not collapse ($(x,y)\notin \Delta_{2d}(T)$). We may suppose that $r>s$ interchanging $x$ and $y$ if necessary. Let $l\geq0$ and $k>0$ ($k=r-s$) such that $T^{k+l}x=T^ly$ are the first iterates of $x$ and $y$ that match. From this point, say $z$, there are $t\geq0$ steps to reach the periodic orbit ($t=s-l=r-l-k$) at a point that we call $z'$ ($z'=T^rx=T^sy$), and then $m-1$ steps ahead along the periodic orbit until it closes (see Figure \ref{fig:casos46A} (left)). Figure \ref{fig:casos46A} (right) shows\!\!
 \thisfloatsetup{floatwidth=\textwidth,capbesidewidth=.15\textwidth}
 \begin{figure}[h]
   \centering
   \includegraphics[width=.9\textwidth]{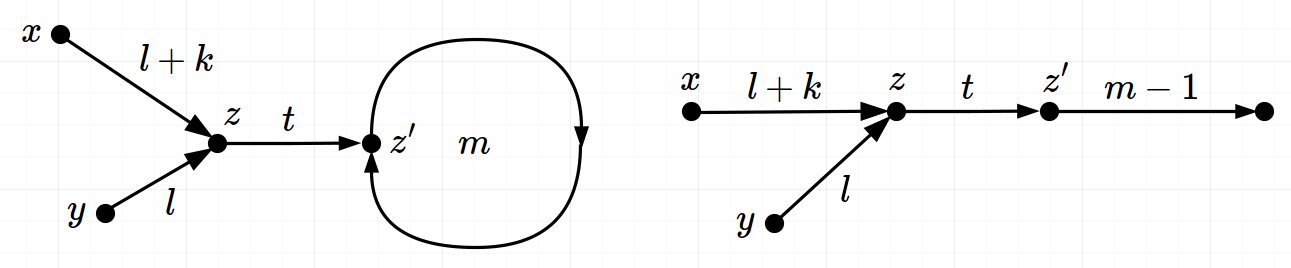}
   \caption{}
   \label{fig:casos46A}
 \end{figure}
 the situation we arrive after deleting the last step before the periodic orbit closes at $z'$ (compare with Figure \ref{fig:caso2}).
 
 As $T$ is a locally injective map there exist compact neighborhoods $U_0$ and $V_0$ of $x$ and $y$ respectively, such that $U_0,\ldots,T^{l+k-1}U_0,$ $V_0,\ldots,T^{s+m-1}V_0$ ($s=l+t$) are pairwise disjoint with $T$ injective on these sets. Let $O=T^{l+k}U_0\cap T^lV_0$, $U=H_{m,r}\cap U_0\cap T^{-l-k}O$ and $V=H_{m,s}\cap V_0\cap T^{-l}O$. We have that $U$ and $V$ are compact, $T^{l+k}U=T^lV$, $U,\ldots,T^{l+k-1}U,$ $V,\ldots,T^{s+m-1}V$ are pairwise disjoint and $T$ is injective on these sets. Then, as $\dim U\leq m-1$ (because $U$ is homeomorphic to $T^rU\subseteq H_m$ and $\dim H_m<\frac m2$) we can apply Lemma \ref{ObsUV_grafica} (with $n=s+m-1$) and conclude that $W=U\times V\cap G_{k,l}=U_0\times V_0\cap G_{k,l}\cap H_{m,r}\times H_{m,s}$ is a neighborhood of $p$ (in the subspace $M_4=H_{m,r}\times H_{m,s}\cap G_{k,l}$) for which $\Omega(W)$ is open and dense ($s+m-1\leq2d$).
 
 {\it Subcase B:} In this subcase the orbits of $x$ and $y$ reach the common periodic orbit at different points $x'=T^rx$ and $y'=T^sy$, respectively (Figure \ref{fig:casos46B} (left)).
 \thisfloatsetup{floatwidth=\textwidth,capbesidewidth=.15\textwidth}
 \begin{figure}[h]
   \centering
   \includegraphics[width=\textwidth]{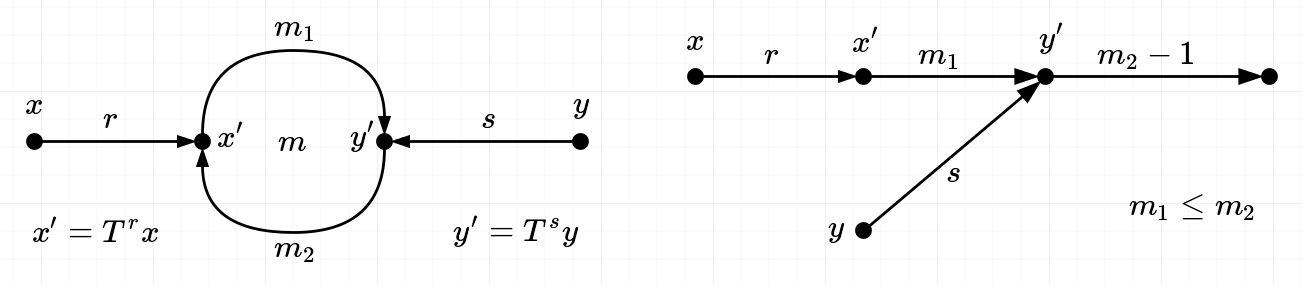}
   \caption{}
   \label{fig:casos46B}
 \end{figure}
 
 Let $m_1>0$ be the number of steps along the periodic orbit form $x'$  to $y'$ and $m_2>0$ number of steps from $y'$ to $x'$ ($m_1+m_2=m$). We have $m_1\geq m/2$ or $m_2\geq m/2$. Interchanging $x$ and $y$ if necessary we can suppose that we are in the second case. Deleting the step in which the periodic orbit closes at $x'$ we arrive to the situation pictured in Figure \ref{fig:casos46B} (right). Comparing $r+m_1$ with $s$ we have two cases: $r+m_1>s$ or $r+m_1<s$ ($x$ and $y$ do not collapse). Suppose we have  $r+m_1>s$ (the other case follows analogously). Let $l\geq0$ and $k>0$ such that $l=s$ and $r+m_1=k+l$. 
 
 As $T$ is a locally injective map there exist compact neighborhoods $U_0$ and $V_0$ of $x$ and $y$ respectively, such that $U_0,\ldots,T^{l+k-1}U_0,$ $V_0,\ldots,T^{s+m_2-1}V_0$ are pairwise disjoint with $T$ injective on these sets. Let $O=T^{l+k}U_0\cap T^lV_0$, $U=H_{m,r}\cap U_0\cap T^{-l-k}O$ and $V=H_{m,s}\cap V_0\cap T^{-l}O$. We have that $U$ and $V$ are compact, $T^{l+k}U=T^lV$, $U,\ldots,T^{l+k-1}U,$ $V,\ldots,T^{s+m_2-1}V$ are pairwise disjoint and $T$ is injective on these sets.
 As $\dim U\leq[\frac{m-1}2]\leq m_2-1$ we can apply Lemma \ref{ObsUV_grafica} (with $n=s+m_2-1]$) and conclude that $W=U\times V\cap G_{k,l}=U_0\times V_0\cap H_{r,m}\times H_{s,m}\cap G_{k,l}$ is a neighborhood of $p$ (in the subspace $M_4=H_{m,r}\times H_{m,s}\cap G_{k,l}$) for which $\Omega(W)$ is open and dense ($s+m_2-1\leq 2d$).
 
 Finally, observe that in both subcases only finitely many subspaces $M_4$ are involved because of the restrictions $0\leq r\leq m+r<2d$, $0\leq s\leq m+s\leq2d$ and $0\leq l<l+k\leq2d$.

 {\sc Case 5:} This case is similar to Case 3. In this case the $2d$-step orbits are disjoint and $\len(y_0^{2d})\leq \len(x_0^{2d})<2d$. The last condition implies that $x\in H_{r,l}$ and $y\in H_{m,k}$ with $0\leq l\leq r+l<2d$, $0\leq k\leq m+k<2d$ and $k+m\leq r+l$. As $T$ is locally injective there exist compact neighborhoods $U_0$ and $V_0$ of $x$ and $y$ respectively, such that $U_0,\ldots,T^{r+l-1}U_0,V_0,\ldots,T^{m+k-1}V_0$ are pairwise disjoint with $T$ injective on these sets. Let $U=U_0\cap H_{r,l}$, $V=V_0\cap H_{m,k}$ and note that $U$ and $V$ are compact, with $\dim U\leq\dim H_r\leq\frac{r-1}2\leq\frac{r+l-1}2$ and similarly $\dim V\leq\frac{m+k-1}2\leq\frac{r+l-1}2$. Then we can apply Lemma \ref{ObsUV_disjuntosNPvsP} (with $n=r+l-1$) to conclude that $\Omega(W)$ is open and dense. Note that $W=U\times V$ is a neighborhood of $p$ in the subspace $M_5=H_{r,l}\times H_{m,k}$. Again in this case, only finitely many subspaces $M_5$ are involved due to the restrictions on $r,l,m$ and $k$.
\end{proof}


\begin{bibdiv}
\begin{biblist}

\bib{AAM}{article}{
author={M. Achigar},
author={A. Artigue},
author={I. Monteverde},
title={Expansive homeomorphisms on non-Hausdorff spaces},
journal={Topology and its Applications},
volume={207}, 
year={2016}, 
pages={109--122}}

\bib{Ae}{article}{
author={D. Aeyels},
title={Generic observability of differentiable systems},
journal={SIAM J. Control and Optimization},
volume={19},
year={1981},
pages={595--603}}

\bib{AH}{book}{
author={N. Aoki},
author={K. Hiraide},
title={Topological theory of dynamical systems},
publisher={North-Holland},
year={1994}}

\bib{Coorn}{book}{
author={M. Coornaert},
title={Topological Dimension and Dynamical Systems},
publisher={Springer},
year={2015}}

\bib{CK}{article}{
author={E. M. Coven},
author={M. Keane},
title={Every compact metric space that supports a positively expansive homeomorphism is finite},
year={2006},
volume={48},
pages={304--305},
journal={IMS Lecture Notes Monogr. Ser., Dynamics \& Stochastics}}


\bib{Gut16}{article}{
title={Takens embedding theorem with a continuous observable},
author={Y. Gutman},
journal={Ergodic Theory, Advances in Dynamical Systems, Edited by Assani, Idris, De Gruyter},
year={2016},
pages={77--107}}

\bib{Gut15}{article}{
author = {Y. Gutman}, 
title = {Mean dimension and Jaworski-type theorems},
volume = {111}, 
number = {4}, 
pages = {831-850}, 
year = {2015},
journal = {Proceedings of the London Mathematical Society}}

\bib{Hi88}{article}{
author={K. Hiraide},
journal={Proceedings of the AMS},
volume={104}, 
year={1988},
title={Positively expansive maps and growth of fundamental groups},
pages={934--941}}

\bib{HW}{book}{
author={W. Hurewicz},
author={H. Wallman},
title={Dimension Theory},
publisher={Princeton Univ. Press}, 
year={1948}}

\bib{Kato93}{article}{
author={H. Kato},
title={Continuum-wise expansive homeomorphisms},
journal={Canad. J. Math.},
volume={45},
number={3},
year={1993},
pages={576--598}}


\bib{Ma}{article}{
author={R. Mañé},
title={Expansive homeomorphisms and topological dimension},
journal={Trans. of the AMS}, 
volume={252}, 
pages={313--319}, 
year={1979}}

\bib{Nerurkar}{article}{
author={M. Nerurkar}, 
title={Observability and Topological Dynamics},
journal={Journal of Dynamics and Differential Equations}, 
volume={3}, 
year={1991},
pages={273--287}}

\bib{Shub69}{article}{
title={Endomorphisms of Compact Differentiable Manifolds},
author={M. Shub},
journal={American Journal of Mathematics}, 
volume={91},
year={1969},
pages={175--199}}

\bib{Takens}{book}{
author={F. Takens},
title={Detecting strange attractors in turbulence},
bookTitle={Dynamical Systems and Turbulence, Warwick 1980: Proceedings of a Symposium Held at the University of Warwick 1979/80},
series={Lecture notes in mathematics},
volume={898},
year={1981},
publisher={Springer Berlin Heidelberg},
pages={366--381}}

\end{biblist}
\end{bibdiv}
\end{document}